\documentclass{article}
\usepackage{amssymb}
\usepackage{lineno}
\usepackage[left=2cm, right=2cm, top=2cm]{geometry}
\usepackage[utf8]{inputenc} 
\usepackage[T1]{fontenc}
\usepackage{hyperref}   
\usepackage{url}
\usepackage[pdftex]{graphicx}
\usepackage{amsthm,amssymb}
\usepackage{thmtools}
\usepackage{environ}
\usepackage{nicefrac}  
\usepackage{siunitx}
\usepackage{tikz}
\usepackage[affil-it]{authblk}
\usepackage{epstopdf}
\usepackage{amsfonts,amssymb,amsbsy}
\usepackage{mathtools}
\usepackage{amsmath}
\usepackage{nccmath}
\usepackage{tabu}
\usepackage{tabularx}
\usepackage{multirow}
\usepackage{microtype} 
\usepackage{placeins}
\DeclareMathOperator{\regret}{regret}
\newtheorem{theorem}{Theorem}[section]

\newtheorem{lemma}[theorem]{Lemma}

\theoremstyle{definition}
\newtheorem{definition}{Definition}[section]
\newtheorem{assumption}{Assumption}[section]

\title{Iterated Piecewise-Stationary Random Functions\footnote{
 R.G. and R.S. are with the Department of Electronics and Electrical Engineering, 
University College Dublin, 
Ireland, Belfield, Dublin 4, \texttt{ramen.ghosh@ucdconnect.ie} and \texttt{robert.shorten@ucd.ie}.
J.M. is with IBM Research -- Ireland, B3 F14, 
Technology Campus Damastown, Mulhuddart, Dublin 15, Ireland,
\texttt{jakub.marecek@ie.ibm.com}.
This work was in part supported by Science Foundation Ireland grant 16/IA/4610.
}
}
\author{Ramen Ghosh, Jakub Marecek, and Robert Shorten}

\usepackage{parskip}

\newenvironment{keywords}
    {
    }

\begin{document}
\setlength{\parskip}{3pt} % 1ex plus 0.5ex minus 0.2ex}
\setlength{\parindent}{6pt}

\maketitle
\begin{abstract}
Within the study of uncertain dynamical systems, iterated random functions
are a key tool. There, one samples a family of functions according to a stationary distribution. Here, we introduce an extension, where one sample functions according to a time-varying distribution over the family of functions. For such iterated piecewise-stationary random functions on Polish spaces, we prove a number of results, including a bound on the tracking error.
\end{abstract}


\section{Introduction}

In the design of distributed systems, a fundamental requirement that underpin many emerging business models is the need to monetize a service. To monetize a service, one needs to define levels of service, and develop resource-allocation strategies that guarantee that these levels are satisfied for individuals making use of that service. 
Roughly speaking, this necessitates the need for resource-allocation strategies that not only maximize (weighted) resource utilisation (i.e., profit), but also allocate the resource to agents in a manner that is predictable (and possibly fair) over certain time scales.
This is often complicated by the fact that many distributed systems of interest either involve humans (for example, when congestion information is sent to road users \cite{marevcek2015signaling,marevcek2016signaling}, who react to this information in their route choice) or strategies that are stochastic in nature (for example, when a randomized route-choice algorithm in a self-driving car determines the route based on congestion information). We refer to such situations as systems made up of ensembles of agents with a {\em probabilistic intent}.

Overall, our goal is to manage the response of such an ensemble of agents so that the system is driven to a predictable and desirable equilibrium. To this end, we need to model such an ensemble of agents. Some human-like behaviours and many complex decision-making strategies can be as modelled by \emph{iterated random functions}  \cite{Diaconis1999,FIORAVANTI2019}. In such systems, agents respond to a signal (for example, a price) or a system state with a probability that may depend on the state of the system (i.e., be ``place-dependent''). Strong results are known for iterated function systems. In particular, ensembles of agents have ergodic properties under assumptions known as contraction on average and very benign further assumptions on the place-dependent probabilities. This is appealing both as a means of capturing the behaviour of ensembles of agents, and as a basis for developing strategies for managing such behaviour at a scale that lend themselves to contract design.

In this paper, we ask how to model an ensemble, which changes over time, perhaps restricted to changing countably many times, with two subsequent ensembles not being too different. One could consider an iterated function system with an infinite number of functions and place-dependent probabilities, but this may make it difficult to derive sharp results. Instead, we model the time-varying stochastic dynamical system by an iterated function system with a finite number of Lipschitz maps on a Polish space with the measure for sampling them changing  countably many times. We call this setting the \emph{iterated piecewise-stationary random functions}. This setting is rather general: most results for iterated function systems are restricted to compact metric spaces, and all results we are aware of assume that the measure for sampling the functions remains stationary throughout.  In this setting, we define an ergodic property, which we call the existence and uniqueness a piecewise-invariant probability measure,
establish a sufficient condition for this property to hold,   and present an estimate of distances between any two subsequent invariant probability measures in a Wasserstein-type metric, leading to a bound on the tracking error in terms of the measure over the state space, and a regret bound.

Our paper is structured as follows. In Section $2$, we present some background and highlight our main contributions. Section $3$ contains relevant mathematical tools, definition, notations, and terminology. Section $4$ contains foundation of the new framework of a piecewise-stationary random iteration of Lipschitz self maps on a Polish space and 
the corresponding discrete-time Markov chain on two time-scales.
Section $5$ contains an estimate of the distance between any two subsequent invariant measures arising from such a process, a tracking error and a regret bound. Finally, in Section $6$, we showcase some computational illustrations.

\section{Background and Contribution}

We represent a dynamical system by $(\mathcal X, P)$ where $\mathcal X$ is a measurable state space, and a measurable map $P:\mathcal X \to \mathcal X$. Intuitively, the ergodicity of a dynamical system reveals its long-term statistical behaviours by relating the dynamical properties of the system to the properties of the evolution of measures of the system. In some situations, complex dynamical systems become much more tractable once when the evolution of measures are studied rather than of points. In what follows, this is the approach that we follow. To that end, let $\mathcal M_p(\mathcal X)$, $\nu$ and $\mathcal B(\mathcal X)$ be the space of all probability measures, a probability measure and a Borel sigma-algebra on $\mathcal X$ respectively, one can naturally define $P^*: \mathcal M_p(\mathcal X)\to \mathcal M_p(\mathcal X) : (P^*\mu)(A)=\mu(P^{-1}(A))$. When investigating the properties of the dynamical system $(\mathcal M_p(\mathcal X), P^*)$, the first relevant question is the study of the fixed points, that is the invariant measures: does there exists a $\nu\in \mathcal M_p(\mathcal X)\text{ such that } \nu(A)=\mu(P^{-1}(A))\quad \forall A\in \mathcal B(\mathcal X)$? Given such an invariant measure $\nu$, one can define the measurable dynamical system $(\mathcal X, P, \nu)$. The invariant measure represents equilibrium states, in the sense that the probabilities of events do not change in time.

Within the study of non-linear stochastic dynamical systems, iterated random functions \cite{Diaconis1999} are a key tool. In this setting, a random dynamical system arises by sampling a family of Lipschitz functions according to a stationary probability distribution and applying the sampled functions. An important characteristic of a stochastic dynamical system is its long term behaviour which can be described in terms of a unique measure that mark out the distribution of the process when the time $t\to \infty$. It is, thus, natural for simulating, and important to investigate that such a measure in fact exists, and is unique. Existence of such a measure is easy to show if the functions involved in the process have some regularity. In particular, if the Lipschitz functions $f_1,\dots, f_m$ with Lipschitz constants $L_1, L_2, \dots, L_m$ are given with probabilities $p_1, p_2,\dots, p_m$, one can define Markov operator $P$ on the space of all family of probability Borel measure on a complete metric space $\mathcal X$ and the condition $\sum\limits p_i L_i<1$ assures the existence of an \emph{invariant} measure, i.e., asymptotic stability of $P$. Even if the process is not contractive, a powerful theorem of ergodicity is available in \cite{breiman1960}. Our objective in this short paper is to introduce an extension towards \emph{iterated piecewise-stationary random functions}, where one samples a family of Lipschitz functions according to a time-varying distribution over the same family of functions. Such extensions are important in applications where the number of agents may change over time (i.e., in almost all smart-city type applications, cf. \cite{marevcek2015signaling,marevcek2016signaling,FIORAVANTI2019}). Specifically, a criterion, implying uniqueness of piecewise-invariant measures in Polish space is developed in this note. This allows for the use of (quasi-) Monte-Carlo approaches in estimating the behaviour in situations, where the system (e.g., numbers or preferences of agents) change over time. Specifically, the contributions of this note are as follows. 
\begin{itemize}
\item We introduce the concept of iterated piecewise-stationary random functions in the general setting of Polish spaces. It can also be thought of as two-time scale homogeneous Markov chain on general state-space.
A Markov operator on finite Borel measures on the state space is associated with the Markov chain and subsequently shown to be Feller. 
\item We suggest a criterion for the existence and uniqueness of a piecewise-invariant measure, an ergodic property of iterated piecewise-stationary random functions.
Essentially, we show that the associated Markov operators are contraction on the space of probability measures under the condition of contraction on averages of the Lipschitz constants.
\item We initiate a study of the tracking error and regret in this setting. This is related to the  rate of convergence or mixing, in absolute terms and relative to the best possible rate achievable, respectively.
\end{itemize}

{\bf Comment:} In our context, it is important to note that when the underlying metric space is compact, or the process evolves on a compact subset of the metric space, the standard way of showing existence and uniqueness of invariant measure is to first construct a positive and invariant linear functional on the space of bounded continuous maps, and then using Riesz representation theorem \cite{Riesz09} to conclude the existence of an invariant measure. In contrast, on an unbounded, complete, separable metric space, this idea breaks down, since a positive functional may not correspond to a measure, i.e., the dual space may not corresponds to the space of signed measures on the space, in general, even if the underlying space is locally compact Hausdorff. There, in Polish space, the concept of uniform tightness and tightness of probability measures plays an important role to investigate the invariant probability measure which is introduced in Section $4.2$ and we refer \cite{Modern97,billing} for further details about the concept.
%mendivil2015time

%but we also show approximations of the piece-wise invariant measures based on convex-optimization techniques, under mild assumptions on the family of functions. 

\section{Mathematical Preliminaries}
We now present some basic results and definitions 
that are necessary for the discussion in the sequel.
\begin{definition}
$(\mathcal X, \rho)$ be a Polish space, i.e., a complete and separable metric space. Let $2^{\mathcal X}$ denote the set of all subsets of $\mathcal X$. Now $\mathcal X$ is equipped with a $\sigma$-algebra, i.e., $\mathcal A\subseteq 2^{\mathcal X}$ such that:
\begin{itemize}
\item[(i)] $\emptyset,\mathcal X\in \mathcal A$;\\
\item[(ii)] $A\in \mathcal A\Rightarrow A^c\in \mathcal A, A^c \text{ denotes the complement of } A$ in $\mathcal X$;
\item[(iii)] $\{A_0, A_1,\dots,\}\subseteq \mathcal A\Rightarrow \bigcup\limits_{n=0}^{\infty} A_n\in \mathcal A \text{ and } \bigcap\limits_{n=0}^{\infty} A_n\in \mathcal A$.
\end{itemize}

Then, the pair $(\mathcal X, \mathcal A)$ is called a measurable Polish space, and elements of $\mathcal A$ are called measurable sets. Furthermore, a function between measurable polish space is measurable if pre-images of measurable sets are measurable sets. 
\end{definition}
Now let us introduce measure and probability measure on the Polish space $(\mathcal X, \rho)$.
\begin{definition} 
Given a measurable Polish space $(\mathcal X, \mathcal A)$ , a measure on $\mathcal X$ is a function $\nu: \mathcal A \to \mathbb [0,\infty)$ that satisfies the following conditions:

\begin{itemize}
\item[(i)] $\nu(\emptyset)=0$;\\
\item[(ii)] $\{A_0, A_1,\dots\}\subseteq \mathcal A$ is a collection of pairwise disjoint sets, then 
\begin{align*}
\nu(\bigcup\limits_{n=0}^{\infty} A_n)=\sum\limits_{n=0}^{\infty} \nu(A_n).
\end{align*}
\end{itemize}
\end{definition}
\begin{definition}
A probability measure on $\mathcal X$ is a function $\nu:\mathcal A\to [0,1]$ that satisfies the following conditions:
\begin{itemize}
\item[(i)] $\nu(\emptyset)=0$;\\
\item[(ii)] $\{A_0, A_1,\dots\}\subseteq \mathcal A$ is a collection of pairwise disjoint sets, then 
\begin{align*}
\nu(\bigcup\limits_{n=0}^{\infty} A_n)=\sum\limits_{n=0}^{\infty} \nu(A_n).
\end{align*}
\item[(iii)] $\nu(X)=1$.
\end{itemize}
\end{definition}

\begin{definition}
$f:(\mathcal X,\mathcal A,\nu)\to \mathcal X$ is called an $\mathcal X$ valued random variable iff $f^{-1}(A)\in \mathcal A$ for every $A\in \mathcal A$. 
\end{definition}

When $\mathcal X$ is considered as measurable space, we ensure that the smallest $\sigma$-algebra which consists of all open sets is endowed naturally to remove any pathology. Also, this will assure that any probability measure defined on it is a Lebesgue–Rokhlin probability space or simply as Lebesgue space and results such as Fubini theorem and regular conditional probabilities are applicable. 
Before proceeding, we now us introduce the following notation and terminology. 
\begin{itemize}
\item
$\mathcal{B}(\mathcal X)$ denotes the Borel $\sigma$ algebra (the smallest $\sigma$ algebra which contains all open sets of $\mathcal X$).
\item $C(\mathcal X,\mathbb R)$ be the Banach space of all continuous function on $\mathcal X$, endowed with the supremum norm $\|\cdot\|_{\infty}$.
\item $C_b(\mathcal X,\mathbb R)$ has the same structure as $C(\mathcal X,\mathbb R)$ but with bounded functions.
\item $\mathcal{M}(\mathcal X)$ denotes the real vector space of all signed finite Borel measures on $ \mathcal X$, containing $\mathcal{M}_{f}(\mathcal X)$ the space of all positive measures.
\item $\mathcal{M}_p(\mathcal X)$ denotes the space of all probability measures on $\mathcal X$ contained in $\mathcal{M}_{f}(\mathcal X)$.
\item $f: \mathcal X\to \mathbb R$ is always continuous, Lipschitz, measurable, and may posses some other properties will be mentioned accordingly.
\item The set of natural number from $1$ to $m$ is denoted by $[1,m]:=\{1,\dots, m\}$.
\item $\mathcal G$ denotes the space of all absolutely continuous functions with $|f'(x)|\le 1$ almost everywhere.
\item
An operator $P: C(\mathcal X, \mathbb R)\to C(\mathcal X, \mathbb R)$ is called Feller, if $P(f(x))\in C_b(\mathcal X, \mathbb R)$ for all $f(x)\in C_b(\mathcal X, \mathbb R)$.
\item 
Finally, we call an operator $P^*:\mathcal{M}_{f}(\mathcal X)\to \mathcal{M}_{f}(\mathcal X)$ Markov if 
\begin{align*}
&\textrm{(i)}& P^*(\lambda_1 \nu_1+\lambda_2 \nu_2)&=\lambda_1 P^*(\nu_1)+\lambda_2 P^*(\nu_2) & \quad \forall \lambda_1, \lambda_2 \ge 0, \nu_1, \nu_2 \in \mathcal{M}_f(\mathcal X)\\
&\textrm{(ii)}& P^*(\mu)(\mathcal X)
&=\mu(X) & \quad \forall \mu\in \mathcal{M}_{f}(\mathcal X).
\end{align*}
\item
For
$\nu\in \mathcal M(\mathcal X)$, $\langle f, \nu \rangle :=\int\limits_{\mathcal X} f d\nu$, for any Feller operator say, $P$, mentioned in this article, its dual will be denoted by $P^*$, which we call Markov operators also, and their action is related by $ \langle Pf, \nu\rangle = \langle f, P^* \nu\rangle$, for all the functions $f$ and the measures $\nu$ defined on suitable spaces which will be mentioned accordingly.
\end{itemize}

\begin{definition}[Iterated random functions
\cite{BarnsleyDemkoEltonEtAl1988,Diaconis1999,steinsaltz1999}]
\label{def:irf}
Given a Polish space $(\mathcal X,\rho)$ and a family of maps $\mathcal{F}=\{f_{\theta}: \mathcal X\to \mathcal X, \theta\in \Theta\}$, an iterated function system on the state space $\mathcal X$ is defined by a measure $\nu$ on $\mathcal{X}$ with respect to a $\sigma$-algebra $\mathcal{A}$ which makes the map $\mathcal X\times \mathcal F\to \mathcal X: (x,f_{\theta})\mapsto f_{\theta}(x)$ measurable, where $\theta$'s are chosen with some probability measures $\mu$ on $\Theta$.
\end{definition}

{\bf Comment:} Note the implicit assumption in Definition \ref{def:irf} is that measure $\mu$ does not change over time. In the next section, we show how to relax this assumption.

\section{Iterated Piecewise-Stationary Random Functions}
\subsection{A Problem Statement}

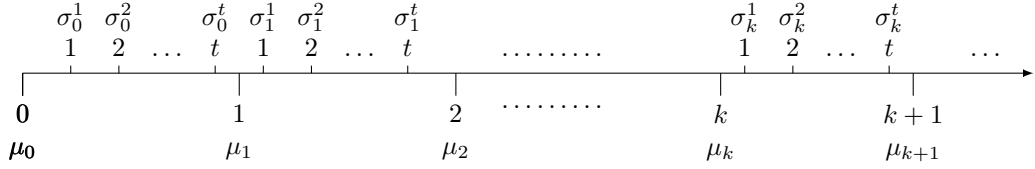
\begin{figure}[tb]
\centering 
\begin{tikzpicture}[x=0.64cm]
\draw[-latex] (0,0) -- (21,0);
\foreach \i [count=\j] in {0,1,5}
 \draw (0,0) -- (0,-.3) node[below,align=center] {0\\$\mu_{{0}}$};
\draw (4.5,0) -- (4.5,-.3) node[below,align=center] {1\\$\mu_{{1}}$};
\draw (9,0) -- (9,-.3) node[below,align=center] {$2$\\$\mu_2$};
\draw (14.5,0) -- (14.5,-.3) node[below,align=center] {$k$\\$\mu_{k}$};
\draw (18.5,0) -- (18.5,-.3) node[below,align=center] {$k+1$\\$\mu_{{k+1}}$};
\foreach \i [count=\j from 0] in {1,5} {
\draw (\i,0) -- (\i,.1) node[above,align=center] {$\sigma_{\j}^{1}$\\1};
\draw (\i+1,0) -- (\i+1,.1) node[above,align=center] {$\sigma_{\j}^{2}$\\2};
\path (\i+2,.1) node[above] {$\ldots$};
\draw (\i+3,0) -- (\i+3,.1) node[above,align=center] {$\sigma_{\j}^{t}$\\$t$};
}
\draw (15,0) -- (15,.1) node[above,align=center] {$\sigma_{k}^{1}$\\1};
\draw (16,0) -- (16,.1) node[above,align=center] {$\sigma_{k}^{2}$\\2};
\path (17,.1) node[above] {$\ldots$};
\draw (18,0) -- (18,.1) node[above,align=center] {$\sigma_{k}^{t}$\\$t$};
\path (20,.1) node[above] {$\ldots$};
\path (11,-.3) node[below] {$\ldots\ldots\ldots$};
\path (11,.1) node[above] {$\ldots\ldots\ldots$};
\end{tikzpicture}
\caption{An illustration of the two time scales, at which we update the measures $\mu_k$ for sampling the functions  and measure $\nu_k^i$ over the state space, respectively. In particular, each application of a function $f_{\sigma_k^i}$, sampled with respect to $\mu_k$, updates the measure $\nu_k^i$  over the state space.}
\end{figure}

We are given a finite set of Lipschitz maps $ \mathcal F=\{f_i: \mathcal X\to\mathcal X\}_{i=1}^{m}$ and a sequence of  probability measures $\{\mu_k\}_{k=0}^{\infty}$ defined on $[1,m]$, where $\mu_k$ denotes the probability measure which is active at time instant $k$ for sampling functions from $\mathcal F$. 
Between $k$ and $k+1$ we can apply $t$-times maps which are chosen from $\mathcal F$ with probability distribution $\mu_k$. Let  $\nu_k\in \mathcal M_p(\mathcal X)$  capture the evolution of state after iterated applications of functions.
Even if we consider the state $X_k$ at time $k$, when a new $\mu_k$ was introduced, to be deterministic, e.g., for $k = 0$, we see that the state after $t$ iterated applications of the functions sampled according to $\mu_k$ is a random variable distributed as:

\begin{align}
\nu_k^t \sim\underbrace{(f_{\sigma_k^{t}}\circ f_{\sigma_k^{t-1}}\circ\cdots\circ f_{\sigma_k^1})}_\text{$t$ iterated applications}(X_k)
\end{align}
where at time step $t$, $\sigma_k^t$ is independent draws from probability measure $\mu_k$ on $[1,m]$, $\circ$ is for composition of functions, and $\mu_k(j):= \mathbb{ P}(\sigma_k^i=j), j\in [1,m]$ for the i.i.d discrete random variables $\sigma_k^1,\sigma_k^2\dots, \sigma_k^i$ taking values in $[1,m]$.

To aid exposition, let us illustrate this with a simple example. In the first step $k = 0$, perhaps starting from a known $X_0^0 \in \mathcal X$, 
with some known $\mu_0 \in \mathcal{M}_p( [1, m] )$, 
we sample a number $\sigma_{0}^{0}$ randomly from $[1,m]$ with probability $\mu_0(\sigma_0^0)$ and move from $X_0^0$ to $X_0^1 \sim f_{\sigma_{0}^1}(X_0^0)$. Next, we select another number $\sigma_0^1$ with support $[1,m]$ with probability $\mu_0(\sigma_0^1)$ to move from $X_1$ to $X_2$ and the process goes on, until some time step $t$, at which we replace our probability measure through which we sample functions from $\mu_0$ to $\mu_1$. 

In the general case, we start from a random variable $X_{k}^i$ supported on $\mathcal X$.
For any Borel set $A\in \mathcal B (\mathcal X)$ we define transitional probability functions, which are probability measure for each fixed $x\in \mathcal X$ and measurable function of $x$ for each fixed $A\in \mathcal{B}(\mathcal X)$, as follows:
\begin{align}
\nu_{k}^{i}(x, A)=\text{ Prob }(X_{k}^{i}\in A| X_k^0=x).
\end{align}

Based on above formulation, we are in a position to define iterated piecewise-stationary random functions as follows:

\begin{definition}[Iterated piecewise-stationary random functions]
Given a Polish space $(\mathcal X,\rho)$, a finite collection of maps $\mathcal{F}=\{f_{i}: \mathcal X\to \mathcal X, i\in [1,m]$, 
an integer $t$ defining a discretisation of time to $(k,i), i \le t$, and an infinite sequence $( \mu_k )_{ k = 1}^{\infty}$ of 
probability measures on $[1,m]$, 
a \emph{iterated piecewise-stationary random functions} on the state space $\mathcal X$ is a stochastic process that evolves ergodically through $\mathcal X$ with evolution equation as follows:
\begin{equation}
\label{tvifs}
X_{k}^{i}=f_{\sigma^{i}_{k}}(X_{k}^{i-1})\quad \text{where } k=0,1,2,3,\dots \text{ and }i=1,2,\dots,t\dots 
\end{equation}
where at time step $i$, $\sigma_k^i$'s are independent draws from probability measure $\mu_k$ on $[1,m]$. 
\end{definition}

Notice that, measure $\mu_k$ over the family of functions changes at certain times $k, k+1, \ldots$,
but in between $k$ and $k+1$, we can perform some number of iterated function applications. For time steps $t$ between $k$ and $k+1$, the conditional distribution of the future depends only on the current state. For the time steps between $k$ and $k+1$, the process $(X_{k}^{t})_t$ is clearly Markovian. The evolution of the measure we sample the functions with, $\forall k, \mu_k$, does not depend on any state. 

\begin{assumption}[Time-varying setting]\label{as:varying}
For time-varying measures $\mu_k$ with which we sample the functions,
$\forall k>0$, the total variation of measures at two subsequent instants $k$ and $k+1$ is never larger than $e$, i.e., 
\begin{align}
\|\mu_k - \mu_{k+1}\|_{TV}=\sum\limits_{j}|\mu_k(j)-\mu_{k+1}(j)| \le e,\forall j \in [1,m], \forall k, j. 
\end{align}
\end{assumption}
That can be seen, in some sense, as a constraint on the changes to be slow-moving. 

\subsection{The Evolution Of The Markov Process}

To investigate the evolution of the above Markov process, introduce the following linear operator $P_k$ when a $\mu_k$ is active. Now, when $\mu_0$ is active, for some $i,j,l\in [1,m]$, if $X_{0}^{0}=x$ then there is probability $\mu_{0}(i)$ that $X_{0}^1=f_{\sigma_0^i}(x)$,
probability $\mu_{0}(j)$ that $X_{0}^1=f_{\sigma_0^j}(x)$, probability $\mu_{0}(l)$ that
$X_{0}^1=f_{\sigma_0^l}(x)$, etc.
Therefore we can write for a general $\mu_k$:
\begin{equation*}
\mathbb{E}[X_{0}^1\mid X_{0}^0=x]= \sum_{j}\mu_k(j) f_{\sigma_k^j}(x)
\end{equation*}
\begin{equation*}
P_kf(x):=\sum_{j}\mu_{k}(j) (f\circ f_{\sigma_k^j})(x)=\mathbb{E}[f(X_{0}^1)\mid X_{0}^0=x]
\end{equation*}
In general, for our setup, define:
\begin{align}
P_k: C(\mathcal X,\mathbb R)\to C(\mathcal X, \mathbb R): 
P_k(f(x))= \mathbb{E}[f(X_{k}^i)\mid X_{k}^{i-1}=x]=\sum\limits_{j=1}^{m} \mu_k(j)(f\circ f_{\sigma_{k}^{j}})(x).
\end{align} 
It is easily noticed that for all $k$, $P_k$ maps any continuous function to a
continuous function which forces the Markov process to be weak Feller.
If $\mathcal X$ were compact, then the dual space of $C(\mathcal X,\mathbb R)$ is the space of all 
signed measures on $\mathcal X$ with total variation norm (by Riesz representation theorem) and the 
adjoint operator of $(5)$ operating on measures on $(\mathcal X, \mathcal{B}(\mathcal X))$ can be deduced as follows:\\
For some measure $\nu\in \mathcal M_f(\mathcal X)$ we have
\begin{align*}
P_k^{*}(\nu)(f)&=\nu (P_kf)=\int P_kf d\nu
=\int \sum_{j=1}^{m} \mu_k(j)(f\circ f_{\sigma_k^j})(x)d\nu
=\sum_{j=1}^{m} \mu_k(j) \int fd(\nu\circ f_{\sigma_{k}^j}^{-1})\\
\end{align*}
Thus, 
\begin{align}
P_k^{*}\nu(A):=\int \mathbb P(X_{k}^{i}\in A| X_{k}^{i-1}=x)d\nu(x)
=\sum\limits_{j=1}^{m} \mu_k(j) \nu (f_{\sigma_k^j}^{-1}(A))
\end{align}

In particular, the adjoint of $(5)$ is, $P^*_k:=\sum\limits_{j=1}^{m}\mu_k(j)\nu\circ f_{\sigma_k^j}^{-1}$.

A central problem is to find conditions for the uniqueness of invariant measures. An iterated function system that has this property is called uniquely ergodic. It might be a very much relevant in terms of practical determination, whether an iterated function system is uniquely ergodic or not, both from a theoretical point and practical standpoint, because that will be easier to conclude simulation results of the process.
An invariant measure for iterated random functions is $\nu_k^*$ for which $P^*_k\nu_k^*=\nu_k^*$. We will show that piecewise-invariant measure exists for the described process 
in Equation \ref{tvifs}, whenever for every $k$, that is between the switching from $\mu_{k-1}$
to $\mu_k$ and switching from $\mu_k$ to $\mu_{k+1}$, there exists an invariant measure $\nu_k^*$, i.e., $P_k^*\nu_k^*=\nu_k^*$.

\section{Main Results}
We now present the main results of this note: namely in Theorem \ref{thm52}, we establish existence of piecewise-invariant measure for the process. In Theorem \ref{thm56}, an estimate of distances between any two subsequent invariant measure is derived. Finally, Theorems \ref{thm:error} and \ref{thm:regret} shows a tracking-error and a loose regret bound on the measures over the state space.

\subsection{The Existence and Uniqueness of Piecewise-Invariant Probability Measure}
\begin{definition}[Uniformly tight measure \cite{bogachev-2007, Modern97, billing}]
An arbitrary $\mathcal M\subseteq \mathcal M_p(\mathcal X)$ is called uniformly tight if $\forall \epsilon >0$
there exists a compact subset $\mathcal K\subseteq \mathcal X$ such that $\nu(\mathcal K)\ge 1-\epsilon\quad \forall \nu\in \mathcal M_{p}(\mathcal X)$.
\end{definition}
It can be shown that on a compact metric space, any family of probability measures is uniformly tight \cite{bogachev-2007, Modern97}  and intuitively, for any other space, probability measures accumulate on compact subsets of the underlying space. 

We now state a standard result related to a uniform tight sequence of measure on Polish space due to Prohorov. 

\begin{theorem}
[Prokhorov \cite{Prokhorov1956ConvergenceOR}] Let $\{\nu_n\}_{n=1}^{\infty} \in \mathcal M_p(\mathcal X)$ be uniformly tight sequence. Then there exists a sub-sequence $\{\nu_{n_k}\}_{k=1}^{\infty}$ of $\{\nu_n\}_{n=1}^{\infty}$ and a $\nu\in \mathcal M_p(\mathcal X)$ such that $\nu_{n_k}\to \nu$ weakly. 
\end{theorem}

Now, with this in mind we establish existence of invariant measures of our transitional probabilities and Markov process described in equation $(2), (3)$.
\begin{theorem}
Let for each $k, P_k$ has weak Feller property on a complete, separable metric space $(\mathcal X,\rho)$. If there exists $x\in \mathcal X$, for which the sequence of transitional probability measures $\{\nu_k^N(x,\cdot)\}_{N\ge 0}$ is uniformly tight, then there exists an invariant probability measure for $P_k^*$.
\label{thm52}
\end{theorem}
\begin{proof}

Assume that there exists at least one $x\in \mathcal X$ for which the sequence $\{\nu_k^j(x, \cdot)\}_{j=0}^{\infty}$ is uniformly tight. Then we show that there exists at least one invariant probability measure for $P_k^*$. The proof is based on the Krylov-Bogoliubov \cite{Kryloff37} type argument. 
Define, a sequence of probability measures on $(\mathcal X, \mathcal B (\mathcal X))$ as follows: 
\begin{align*}
\text{ for } A\in \mathcal B (\mathcal X),\quad 
\nu^N_{k}(A)=\frac{1}{N}\sum\limits_{j=0}^{N-1}\nu_k^j (x,A), \text{ for some fixed } x\in \mathcal X. 
\end{align*}
It is clear that this sequence is also tight, so it has a sub-sequence that converges weekly to some probability measure $\nu_k^*$ on $\mathcal X$. We also have the following equality:
\begin{align}
P_k^*\nu^N_{k}- \nu^N_{k}= \frac{1}{N}\sum\limits_{j=2}^{N+1}\nu_k^j (x,A)-\frac{1}{N}\sum\limits_{j=1}^{N}\nu_k^j (x,A)=\frac{1}{N}[\nu_k^{N+1} (x,\cdot)-\nu_k^1 (x,\cdot) ]
\end{align}

Notice that for each fixed $x\in \mathcal X$, $\nu_k(x,A)$ is a probability measure and the integral of $f(x)$ with respect to such measure is expressed as $\int f(y)\nu_k(x,dy)$, and the interpretation holds for any $N\in \mathbb N$ and written as  $\int f(y)\nu_k^{N}(x,dy)$.
Take any $f\in C_b(\mathcal X, \mathbb R)$ such that $|f(x)|<1$. Fix an $\epsilon >0$. Weak convergence of the probability measures $\{\nu_k^N\}_{N\ge 1}$ ensures that there is a natural number $N>\frac{1}{\epsilon}$ for which
\begin{equation*}
\left |\int f(x) \nu_k^N(dx)-\int f(x) \nu_k^*(dx)\right|\le \epsilon.   
\end{equation*}
Since $(P_kf)$ is continuous, we can chose large $N$ for which 
\begin{equation*}
\left |\int (P_kf)(x) \nu_k^N(dx)-\int (P_kf)(x) \nu_k^*(dx)\right|\le \epsilon.   
\end{equation*}
Now, 
\begin{align*}
&\left|\int f(x) (P_k^*\nu_k^*)(dx)-\int f(x)\nu_k^*(dx)\right| \le \left |\int f(x) (P_k^*\nu_k^*)(dx)
-\int f(x)(P_k^*\nu_k^N)(dx)\right|\\
+&\left |\int f(x) (P_k^*\nu_k^N)(dx)-\int f(x)\nu_k^N(dx)\right|+ \left|\int f(x)\nu_k^N(dx)-\int f(x)\nu_k^*(dx)\right|\\
\le & \left|\int (P_kf)(x)\nu_k^*(dx)-\int (P_k f)(x)\nu_k^N(dx)\right|+ \frac{1}{N}\left |\int f(y)(\nu_k^{N+1}(x,dy)-\int f(y)\nu_k(x,dy)\right|+\epsilon\\
\le & 2\epsilon+ \frac{2}{N}\le 4\epsilon
\end{align*}
Since, the above relation is true for any arbitrary $\epsilon$, we can conclude
\begin{equation*}
\left|\int f(x) (P_k^*\nu_k^*)(dx)-\int f(x)\nu_k^*(dx)\right|=0
\end{equation*}
Also, the $f\in C_b(\mathcal X, \mathbb R)$ was arbitrary too, which forces $P_k^*\nu_k^*=\nu_k^*$.
\end{proof}

The following is a standard way of checking an arbitrary subset of probability measure is uniformly tight: 

\begin{theorem}[Bogachev \cite{bogachev-2007}]
An arbitrary $\mathcal M\subseteq \mathcal M_p(\mathcal X)$ is uniformly tight if there exists a Borel measurable function $g:\mathcal X\to \mathbb R^{+}$ such that for any $c<\infty$,\\
$(i)$ the sets $S_c=\{x: g(x)\le c\}$ are compact,\\
$(ii)$ $\nu(\{x: g(x)=\infty\})=0\forall \nu\in \mathcal M$, and\\
$(iii)$ $\sup\limits_{\nu\in \mathcal M} \int\limits_{X} g d\nu <\infty$. \\
\end{theorem}
The proof of the above theorem precisely shows existence of an Lyapunov function for the system which evolves on Polish space which is not necessarily locally compact. Let us now briefly introduce Lyapunov function for transitional probability on a Polish space.

\begin{definition}
A Borel measurable function $g: \mathcal X\to \mathbb R^{+}$ is called a Lyapunov function for a transitional probability $\nu(x,A)$ if the following conditions are satisfied:\\
$(i)$ $ g^{-1}(\mathbb R^{+})\ne \phi$\\
$(ii)$ the sets $S_c=\{x: g(x)\le c\}$ are compact,\\
$(iii)$ There exists $0\le c_1<1$ and $c_2\in \mathbb R$ such that $\int\limits_{X} g(y) d\nu(x, dy) <\infty$. \\
\end{definition}

This allows one to check that if a transitional probability is Feller and admits a Lyapunov function, then it also has an invariant probability measure $\cite{Hairer2018}$.

\subsection{An Estimate of the Distance between two Invariant Measures}
The state at time $k$ is a random variable, so we will work with  measures $\nu_k^t \in \mathcal{M}_p( \mathcal X )$, which we obtain by sampling from the set $\mathcal{F}$ according to $\mu_k$ and applying the sampled function, both repeated $t$ times to what had been obtained previously (with $\mu_{k-1}, \mu_{k-2}$). In particular, how far is the measure $\nu_k^t$ obtained after $t$ iterated function applications from the piecewise-invariant measure $\nu_k^*$, which we would reach eventually.
A metric $d$ on $\mathcal M_p(\mathcal X)$ defined by 
\begin{align*}
d(\nu_1, \nu_2)= \|\nu_1-\nu_2\|_{\text{TV}}=\sup_{A \in \mathcal B(\mathcal X)} \{|\nu_1(A)- \nu_2(A)|\}=\frac12\sup_f\{\left|\int f\:{\rm d}(\nu-\mu)\right|:  f:\mathcal X\to \mathbb R \textrm{ is measurable}, \|f\|\le 1\}
\end{align*}
(i.e., the total-variation distance between $\nu_1, \nu_2$), makes it a complete metric space when $\mathcal X$ is complete $\cite{bogachev-2007}$. It is standard to show convergence to invariant measure in total variance distance in general theory of Markov chain, but in our situation, the initial distribution can be chosen in way, such that the process does not converge at all. If the process evolves on some compact subset of $\mathbb R^n$, one could choose the initial distribution to be singular with regard to Lebesgue measure if the invariant measure is absolutely continuous and vice-versa. By using a weaker metric such as Kantorovich metric, one can avoid the problem. Convergence in Kantorovich–Rubinstein metric is equivalent to weak convergence. Let us now define dual version of Wasserstein-type metric (also known as Kantorovich–Rubinstein metric) $d$ on $\mathcal M_p(\mathcal X)$ as follows, for any two $\nu_1, \nu_2$: 
\begin{align}
d(\nu_1, \nu_2)=\sup\limits_{f\in \mathcal F_{\alpha}}\left[\int\limits_{\mathcal X} fd\nu_1-\int\limits_{\mathcal X} f d\nu_2\right]= \sup\limits_{f\in \mathcal F_{\alpha}}\int\limits_{\mathcal X} fd(\nu_1-\nu_2), 
\end{align}
where for $\alpha\in (0,1]$ and
\begin{align}
\mathcal F_{\alpha}=\{f : \mathcal X\to \mathbb R: |f(x)-f(y)|\le d^{\alpha}(x,y)\}.
\end{align}

\begin{lemma}
If for some $r\in (0,1), \sum\limits_{j=1}^{m}\mu_k(j) |(f\circ f_{\sigma_{k}^{j}})(x)-(f\circ f_{\sigma_{k}^{j}})(y)|^{\alpha}\le r d^{\alpha}(x,y)$, then $\frac{1}{r} (P_kf)\in \mathcal F_{\alpha}\forall f\in \mathcal F_{\alpha}$, and $P_k^*$ are strictly contractive i.e., $\forall\hspace{0.2cm} \nu_1, \nu_2\in \mathcal M_p(\mathcal X)$ 
\begin{align}
d(P_k^*\nu_1, P_k^* \nu_2)\le r d(\nu_1, \nu_2)
\end{align}
\end{lemma}
\begin{proof}
\begin{align*}
\left |\frac{1}{r}(P_kf)(x)-\frac{1}{r}(P_k)f(y)\right|&=\left|\frac{1}{r}\sum\limits_{j=1}^{m} \mu_k(j)(f\circ f_{\sigma_{k}^{j}})(x) -\sum\limits_{j=1}^{m} \mu_k(j)(f\circ f_{\sigma_{k}^{j}})(y)\right|\\
&\le \frac{1}{r}\sum\limits_{j=1}^{m} \mu_k(j)\left |(f\circ f_{\sigma_{k}^{j}})(x)-(f\circ f_{\sigma_{k}^{j}})(y)\right|\\
&\le d^{\alpha}(x,y)
\Rightarrow \frac{1}{r}(P_kf)\in \mathcal F_{\alpha}.
\end{align*}

Now,
\begin{align*}
d(P_k^*\nu_1, P_k^* \nu_2)&=\sup\limits_{f\in \mathcal F_{\alpha}}\left [ \int f d(P_k^*\nu_1)- \int fd(P_k^*\nu_2)\right ]\\
&=\sup\limits_{f\in \mathcal F_{\alpha}}\left [ \int (P_kf)d\nu_1- \int (P_kf)d\nu_2\right ]\\
&=\sup\limits_{f\in \mathcal F_{\alpha}}\left [ \int (P_kf)d(\nu_1-\nu_2)\right ]\\
&=r\cdot\sup\limits_{f\in \mathcal F_{\alpha}}\left [ \int \left(\frac{1}{r}P_kf\right)d(\nu_1-\nu_2)\right ]\\
&= r\cdot\sup\limits_{g\in \mathcal F_{\alpha}} \int gd(\nu_1-\nu_2)\left [\because g=\frac{1}{r}P_kf\in \mathcal F_{\alpha}\right ]\\
&\le r d(\nu_1, \nu_2).
\end{align*}
\end{proof}

\begin{lemma}
\label{lem55}
$P_k^*$ is a contractive Markov operator in Wasserstein type metric if the expected modulus of Lipschitz continuity with respect to $\mu_k$ is less than $r$, i.e., $\sum\limits_{j=1}^{m} \mu_k(j)f_{\sigma_k^j}'(x)\le r<1$ almost everywhere.
\end{lemma}

\begin{proof}
Call a Markov operator $P$ contractive with contraction factor $r\in (0,1)$, and non-expansive when $r=1$, if $d(P\nu_1, P\nu_2)\le r d(\nu_1, \nu_2)\quad \forall \nu_1,\nu_2\in \mathcal M_p(\mathcal X)$. Let us first find for any $i,j$, the following estimates of the distance between any two arbitrary measures:

\begin{align}
d(P_k^*\nu_k^i,P_k^*\nu_{k}^j)
=& \sup\limits_{f\in \mathcal G}\left[\int f d(P_k^*\nu_k^i)-\int f d(P_k^*\nu_k^j)\right]\nonumber\\
=& \sup\limits_{f\in \mathcal G}\int (P_k f)d(\nu_k^i-\nu_k^j)\nonumber\\
=&\sup\limits_{f\in \mathcal G}\int \left (\sum\limits_{j=1}^{m} \mu_k(j)(f\circ f_{\sigma_{k}^{j}})(x) \right )d(\nu_k^i-\nu_k^j)\nonumber\\
=&r\sup\limits_{f\in \mathcal G}\int \mathcal{J}_{f}(x) d(\nu_k^i-\nu_k^j),
\end{align}
where $\mathcal{J}_{f}(x)=\frac{1}{r}\sum\limits_{j=1}^{m} \mu_k(j)(f\circ f_{\sigma_{k}^{j}})(x)\Rightarrow |\mathcal{J}_{f}'(x)|\le \frac{1}{r}\sum\limits_{j=1}^{m} \mu_k(j)|f'(f_{\sigma_{k}^{j}}(x))| |f_{\sigma_{k}^{j}}'(x)|\le 1\text{ almost everywhere}$. Thus,  $\mathcal J_f\in \mathcal G$ and $(11)$ becomes
\begin{align}
d(P_k^*\nu_k^i,P_k^*\nu_{k}^j)\le r d(\nu_k^i,\nu_{k}^j).
\end{align}
\end{proof}

Using Lemma \ref{lem55}, an estimate of the distance between an arbitrary measure and the invariant measure, can be obtained as follows:
\begin{align}
d(\nu, \nu^*)&\le d(\nu, P_k^*\nu)+d(P_k^*\nu, \nu^*)\nonumber\\
&=d(\nu, P_k^*\nu)+d(P_k^*\nu, P_k^*\nu^*)\nonumber\\
&\le d(\nu, P_k^*\nu)+rd(\nu, \nu^*) \nonumber\\
&\Rightarrow d(\nu, \nu^*)\le \frac{1}{1-r} d(\nu, P_k^* \nu) 
\end{align}

When the process reaches an invariant measure associated with $\mu_k$, switches to $\mu_{k+1}$, and reaches another invariant measure, it is natural to ask how far the two invariant measures $\nu_k^*$ and $\nu_{k+1}^*$ are. We have an estimate: 
\begin{theorem}
For all $k$, and for some $r\in (0,1)$, there is a bound on the distance of two subsequent invariant measures $\nu_k^*, \nu_{k+1}^*$:
\begin{align}
d(\nu_k^*, \nu_{k+1}^*)&\le \frac{1}{1-r}\left (M\sum\limits_{j}\Vert f_{\sigma_k^j}(x)-f_{\sigma_{k+1}^j}(x)\Vert_{\infty}+Be\right )
\end{align}
where $B$ is a bound for the functions $f\in \mathcal F_{\alpha}$, $e$ is the bound on the rate of change in Assumption \ref{as:varying}, and $M$ is a bound on the derivatives of $f\in \mathcal F_{\alpha}$.
\label{thm56}
\end{theorem}

\begin{proof}
Consider $P_k^*$ be the Markov operator of the IFS when $\mu_k$ is active with invariant measure $\nu_k^*$ and $P_{k+1}^*$ be the Markov operator of the IFS when $\mu_{k+1}$ is active with invariant measure $\nu_{k+1}^*$.
Let us first notice that
\begin{align}
d(\nu_k^*, \nu_{k+1}^*)=&d(P_k^*\nu_k^*, P_{k+1}^*\nu_{k+1}^*)\nonumber\\
\le & d(P_k^*\nu_k^*, P_k^*\nu_{k+1}^*)+d(P_k^*\nu_{k+1}^*, P_{k+1}^*\nu_{k+1}^*)\nonumber\\
\le & rd(\nu_k^*,\nu_{k+1}^*)+d(P_{k}\nu_{k+1}^*, P_{k+1}\nu_{k+1}^*) \quad[\because (10)]\nonumber\\
\Rightarrow d(\nu_k^*, \nu_{k+1}^*)&\le \frac{d(P_{k}\nu_{k+1}^*, P_{k+1}\nu_{k+1}^*)}{1-r}
\end{align}

Now, secondly, notice that:
\begin{align}
\left \|P_kf(x)-P_{k+1}f(x) \right\|=&\|\sum\limits_{j}\mu_k(j)(f\circ f_{\sigma_k^j})(x)-\sum\limits_{j}\mu_{k+1}(j)(f\circ f_{\sigma_{k+1}^j})(x)\|\nonumber\\
=&\| \sum\limits_{j}[\mu_k(j)-\mu_{k+1}(j)](f\circ f_{\sigma_{k+1}^j})(x)+ \sum\limits_{j}\mu_k(j)[(f\circ f_{\sigma_k^j})(x)-(f\circ f_{\sigma_{k+1}^j})(x)]\|\nonumber\\
\le& \sum\limits_{j}|\mu_k(j)-\mu_{k+1}(j)|\Vert(f\circ f_{\sigma_{k+1}^j})(x)\Vert+ M\sum\limits_{j}\Vert f_{\sigma_k^j}(x)- f_{\sigma_{k+1}^j}(x)\Vert\nonumber\\
\le&B \sum\limits_{j}|\mu_k(j)-\mu_{k+1}(j)|+ M\sum\limits_{j}\Vert f_{\sigma_k^j}(x)-f_{\sigma_{k+1}^j}(x)\Vert\nonumber\\
\le&B \|\mu_k-\mu_{k+1}\|_{\text{TV}}+ M\sum\limits_{j}\Vert f_{\sigma_k^j}(x)-f_{\sigma_{k+1}^j}(x)\Vert\nonumber\\
\le& Be+ M\sum\limits_{j}\Vert f_{\sigma_k^j}(x)- f_{\sigma_{k+1}^j}(x)\Vert_{\infty}\quad [ \because (4)]
\end{align}
And, finally
\begin{align}
d(P_{k}\nu_{k+1}^*, P_{k+1}\nu_{k+1}^*)&= \sup\limits_{f}\int f d(P_k^*\nu_{k+1}^*-P_{k+1}^*\nu_{k+1}^*)= \sup\limits_{f}\int (P_kf-P_{k+1}f)d(\nu_{k+1}^*)\nonumber\\
&\le M\sum\limits_{j}\Vert f_{\sigma_k^j}(x)-f_{\sigma_{k+1}^j}(x)\Vert_{\infty}+Be.\nonumber\\
\Rightarrow d(\nu_k^*, \nu_{k+1}^*)&\le \frac{1}{1-r}\left (M\sum\limits_{j}\Vert f_{\sigma_k^j}(x)-f_{\sigma_{k+1}^j}(x)\Vert_{\infty}+Be\right )
\end{align}
\end{proof}

\subsection{A Tracking Error and a Regret Bound}

With the technical results of the previous subsection, we can start studying the distance of the measure $\nu_k^1$ after $k$ changes in the measure $\mu_k$
 to the invariant measure $\nu_k^*$ after $k$ changes in the measure $\mu_k$,
 which we call the tracking error.
%subsequently the  of the operator $P_k^*$ and the best possible tracking error

\begin{theorem}[Tracking error in the on-line setting]
\label{thm:error}
Let Assumptions~\ref{as:varying} hold. 
For all $k$, and 
for some $r\in (0,1)$,
we have convergence to $\nu_k^* \in\mathcal{M}_p(\mathcal X)$ in the sense that 
\begin{equation}\label{eq.asympt_error}
d(\nu_k^1, \nu_k^*) \le \left [\frac{r}{(1-r)^2}\right ] d(\nu_k^{1}, \nu_k^{0}).
\end{equation}
\end{theorem}
\begin{proof}
Define for iteration, $P_k^*(\nu_k^{i-1})=\nu_k^{i}$, then by induction on $i$ and due to contractive property of $P_k^*$ it is easy to see for any $i$ we have 
\begin{align}
d(\nu_k^{i+1}, \nu_k^i)\le r^i d(\nu_k^1, \nu_k^0)
\end{align}
Thus for any $j>i$, we have the following inequality 
\begin{align}
d(\nu_k^j, \nu_k^i)&\le d(\nu_k^j, \nu_k^{j-1})+d(\nu_k^{j-1}, \nu_k^{j-2})+\dots+d(\nu_k^{j-(j-i-1)}, \nu_k^{j-(j-i)})\nonumber \\
&\le r^{j-1}d(\nu_k^1, \nu_k^0)+r^{j-2}d(\nu_k^1, \nu_k^0)+\dots+r^i d(\nu_k^{1}, \nu_k^{0}) \quad (\because (19))\nonumber\\
&= \left (\sum\limits_{s=1}^{j-i-1} r^s\right )\times r^i\times d(\nu_k^{1}, \nu_k^{0})\le \left (\sum\limits_{s=1}^{\infty} r^s\right )\times r^i\times d(\nu_k^{1}, \nu_k^{0})\nonumber\\
&= \left (\frac{1}{1-r}\right )\times r^i\times d(\nu_k^{1}, \nu_k^{0}).
\end{align}
Combining $(13)$ and $(20)$ we get 
\begin{align}
d(\nu_k^1, \nu_k^*)&\le \left [\frac{r}{(1-r)^2}\right ] d(\nu_k^{1}, \nu_k^{0})
\end{align}
and that also holds in the large limit of $k$.
\end{proof}

Let us now consider the difference tracking error of the operator $P_k^*$ and the best possible tracking error across a family ${\mathcal G}$ of operators. This is sometimes known as the regret:
\begin{definition}
\begin{align}
\regret(P_k^*) := 
\sup_{g \in \mathcal G}
\left( \sum\limits_{k=0}^{\infty}\sum_{s = 1}^{t} d(\nu_k^s, \nu_k^*)
- \sum\limits_{k=0}^{\infty}\sum_{s = 1}^{t} d(g \nu_k^s, \nu_k^*)
\right).
\end{align}
\end{definition}

Intuitively, the regret is zero, if our convergence is the fastest possible.
We can bound regret as follows:

\begin{theorem}[Regret bound in the on-line setting] 
\label{thm:regret}
Let Assumptions~\ref{as:varying} hold. For $t$ iterated function applications after each change at $k$, we have:
\begin{align}
\regret(P_k^*) %\le \sum_{k=0}^{\infty}\sum_{s = 1}^{t} d(\nu_k^s, \nu_k^*)
\le \sum_{k} \left[\frac{2r}{(1-r)^2}\right]d(\nu_k^1, \nu_k^0).
\end{align}
\end{theorem}
\begin{proof}
\begin{align*}
\regret(P_k^*) & \le \sum_{k=0}^{\infty}\sum_{s = 1}^{t} d(\nu_k^s, \nu_k^*) \nonumber\\
&= \sum \left[ d(\nu_k^t,\nu_k^*)+ d(\nu_k^{t-1}, \nu_k^*)+\dots+ d(\nu_k^1, \nu_k^*)\right]\nonumber\\
&\le \sum \left[ d(\nu_k^t,P_k^*\nu_k^t)+ d(\nu_k^{t-1}, P_k^*\nu_k^{t-1})+\dots+ d(\nu_k^1, P_k^*\nu_k^{1})\right] \quad (\because 13)\nonumber\\
&\le \sum \left[\frac{r^{t}}{1-r}+\frac{r^{t-1}}{1-r}+\dots+\frac{r}{1-r}+\frac{r}{(1-r)^2}\right]d(\nu_k^1, \nu_k^0)\nonumber\quad (\because 20, 21)\\
&= \sum \left[\frac{r}{1-r}\left(r^{t-2}+\dots+r+1\right)+\frac{r}{(1-r)^2}\right]d(\nu_k^1, \nu_k^0)\nonumber\\
&\le \sum \left[\frac{r}{1-r}\left(\sum\limits_{t=1}^{\infty}r^{t}\right)+\frac{r}{(1-r)^2}\right]d(\nu_k^1, \nu_k^0)\nonumber\\
&= \sum \left[\frac{2r}{(1-r)^2}\right]d(\nu_k^1, \nu_k^0)\nonumber\\
\end{align*}
\end{proof}

We stress that this bound is only rudimentary and can be improved upon.

\section{A Computational Illustration}

Let us illustrate the main features of our results in simulations related to the famous maple-leaf fractal of Barnsley.
In particular, we consider four two-dimensional iterated random functions:
\begin{align*}
f_1:& \begin{bmatrix}x\\y\end{bmatrix}\mapsto \begin{bmatrix}0.8&0\\0&0.8\end{bmatrix}\begin{bmatrix}x\\y\end{bmatrix}+\begin{bmatrix}0.1\\0.04\end{bmatrix},\\
f_2:& \begin{bmatrix}x\\y\end{bmatrix}\mapsto\begin{bmatrix}0.5&0\\0&0.5\end{bmatrix}\begin{bmatrix}x\\y\end{bmatrix}+\begin{bmatrix}0.25\\0.4\end{bmatrix},\\
f_3:& \begin{bmatrix}x\\y\end{bmatrix}\mapsto \begin{bmatrix}0.355&0.355\\0.355&-0.355\end{bmatrix}\begin{bmatrix}x\\y\end{bmatrix}+\begin{bmatrix}0.266\\0.078\end{bmatrix},\\
f_4:& \begin{bmatrix}x\\y\end{bmatrix}\mapsto \begin{bmatrix}0.355&-0.355\\-0.355&-0.355\end{bmatrix}\begin{bmatrix}x\\y\end{bmatrix}+\begin{bmatrix}0.378\\0.434\end{bmatrix}. 
\end{align*}
There are 3 different probability distribution $\mu_0, \mu_1, \mu_2$ for sampling from the index set $\{1,2,3,4\}$ of the four functions, as summarized in Table \ref{ex-prob}. Notice that $\{\mu_k\}_{k=0}^{2}$ are chosen such that Assumption \ref{as:varying} is satisfied. Our simulations are performed in Matlab and $30,000$ iterations between each change in $\mu_i$, i.e., from $\mu_0$ to $\mu_1$ and from $\mu_1$ to $\mu_2$, giving a total of $90,000$ points in each sample path. 

\begin{table}[h!]
\centering
\caption{Time-varying measures $\mu_k$ for our illustrative example.}
\label{ex-prob}
\begin{tabular}{|m{2cm}|m{2cm}|m{2cm}|m{2cm}|m{2cm}|m{2cm}|}
\hline
\multicolumn{2}{|c|}{Iteration with $\mu_0$}&\multicolumn{2}{|c|}{Iteration with $\mu_1$ }&\multicolumn{2}{|c|}{Iteration with $\mu_2$}\\
\hline
Functions & Probabilities & Functions & Probabilities & Functions & Probabilities \\
\hline
$f_1$& 0.23& $f_1$ & 0.5& $f_1$ & 0.3\\
$f_2$& 0.22&$f_2$&0.2&$f_2$&0.1\\
$f_3$&0.22&$f_3$&0.2&$f_3$&0.4\\
$f_4$&0.33&$f_4$&0.1&$f_4$&0.2\\
\hline
\end{tabular}
\end{table}

First, let us illustrate iterated random functions 
with a stationary measure 
$\mu_0, \mu_1, \mu_2$, respectively.
Figure \ref{approxall} depicts the set obtained by iterated applications 
of the four Lipschitz transformations on $\mathbb R^2$,
with each of the three distributions $\mu_0, \mu_1, \mu_2$, 
plotted in red, green, and blue, 
respectively. 
The results depicted are in accordance with Theorem \ref{thm52}, 
which shows the existence of piecewise-invariant measure for the 
time-varying iterated function system,
which would be supported on the boundary of the red, green, and blue, 
``maple leaves''.

Next, let us demonstrate the measures over the state space.
Figures \ref{approxNu0}--\ref{approxNu2} depict the probability
 density functions for the $2$-dimensional sample data points,
 which are generated by iterations of the transformations by
  $\mu_0, \mu_1$ and $\mu_2$, respectively. These are, indeed, 
  approximated by a histogram, which suitably discretizes the 
  2 dimensions.
Figure \ref{nuCompared} shows the impact of changes in the measure
 for sampling function with from $\mu_0$ (red) to $\mu_1$ (green) 
 and then $\mu_1$ (green) to $\mu_2$ (blue)
on the distribution over the 2-dimensional states.
Clearly, there are noticeable differences, as well as noticeable 
similarities. 

Finally, let us demonstrate the evolution of the Wasserstein distance between the subsequent measures over the state space and the Wasserstein distance of these measures to the invariant measure.
Figure \ref{nuWasser1} depicts the pairwise Wasserstein distance between the subsequent measures over the state space on our example. The plot shows that the distances between subsequent measures decay, as predicted by Theorem \ref{thm:error}.
Figure \ref{nuWasser2} illustrates the decay of distance to the invariant measure, as predicted by Theorem \ref{thm:error}. 
While \ref{thm:regret} bounds the rate of the decay only loosely, the actual rate of decay seems convincingly fast in this example.

\section{Conclusions}

We have introduced iterated piecewise-stationary random functions, a class of stochastic dynamical system on a Polish space. This extends iterated random functions, where there is a finite family of Lipschitz self maps and where the following point in the time series is determined by the previous one and the random choice of a function from a given family.
The extension considers the random variable directing the choice of the function to change after a certain number of function applications. That is, the random variable is time-varying, but piecewise identically distributed. 
This is the only ``time-varying'' extension of iterated random functions we are aware of, outside of Mendivil's \cite{mendivil2015time}, wherein all functions in the family change into identity over time. 
The long-term asymptotic behaviour of this class of process has been described and
we have shown the existence and uniqueness of the so-called piecewise-invariant measure (an ergodic property). 
In particular, in the complete and separable metric space set up, we have shown conditions for the existence of a piecewise-invariant measure in this new setting of piecewise-stationary random functions. A bound of the distance between two subsequent invariant measures has been shown, leading to  a bound on the tracking error and, admittedly loose, regret bound. Direct applications of this model can be found in modeling of computer networks \cite{corless2016aimd}, control of switched systems \cite{664150}, queuing theory \cite{IFS2001}, computer graphics \cite{barnsley1993}, and mathematical biology  \cite{Lasota1999,wojewodka2013exponential}. 

More broadly, we argue for the study of conditions, under which ensembles of agents that can be modeled as iterated function systems are ergodic under various feedback configurations. 
In this, we are motivated by examples from application domains, where such agents act under the influence of a dynamic price. 
As we have mentioned repeatedly, the issue of ergodicity, which is akin to reproducibility in resource-allocation problems, is at the heart of our concerns. 
That is, the resource-allocation to an individual agent can be reliably replicated across the ensemble, independent of the initial state of the agent.
This is clearly a fundamental requirement for service providers, who wish to issue contracts to individual agents, but operate systems wherein stochastic allocation policies are deployed. 
Our initial work \cite{FIORAVANTI2019} has shown that simple control algorithms (PI, I) do not often give rise to ergodic feedback systems, and hence, they cannot be used for contract design. 
In our current paper, we have extended the class of systems, wherein ergodic feedback systems can be designed.
This opens up a rich set of theoretical questions that are also of great practical importance. For example, questions such as preservation of ergodicity under feedback, the effect of interconnecting several ergodic populations, and dealing with more complicated time-varying resource-allocation problems, are all problems with a both great practical and theoretical appeal.

\bibliographystyle{IEEEtran}
\bibliography{References}

\FloatBarrier
\begin{figure}[htp]
\centering 
% [trim = 14mm 100mm 18mm 5mm, clip, width=14.35cm]
\includegraphics[width=14.35cm]{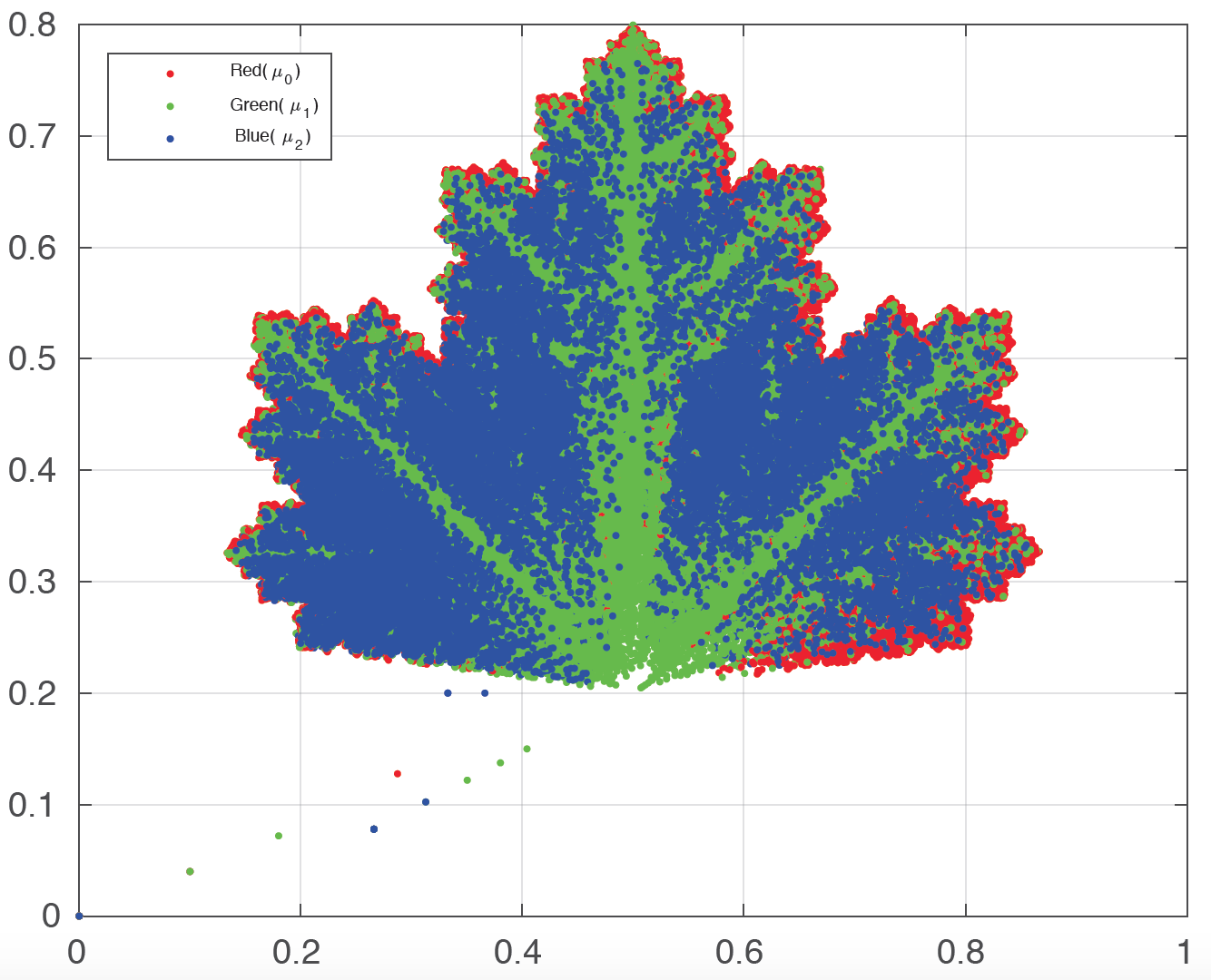}
\caption{A finite-sample approximation 
of the support of the measure over the 2-dimensional state-space 
generated using $\mu_0$ (in red), $\mu_1$ (in green), and $\mu_2$ (in blue), respectively, 
as the measure for sampling the 4 functions. 
Notice that only the boundary of the support corresponds to the
 support of the invariant measure.}
\label{approxall}
\end{figure}

\begin{figure}[htp]
\centering 
\includegraphics[trim = 14mm 100mm 18mm 5mm, clip, width=14.35cm]{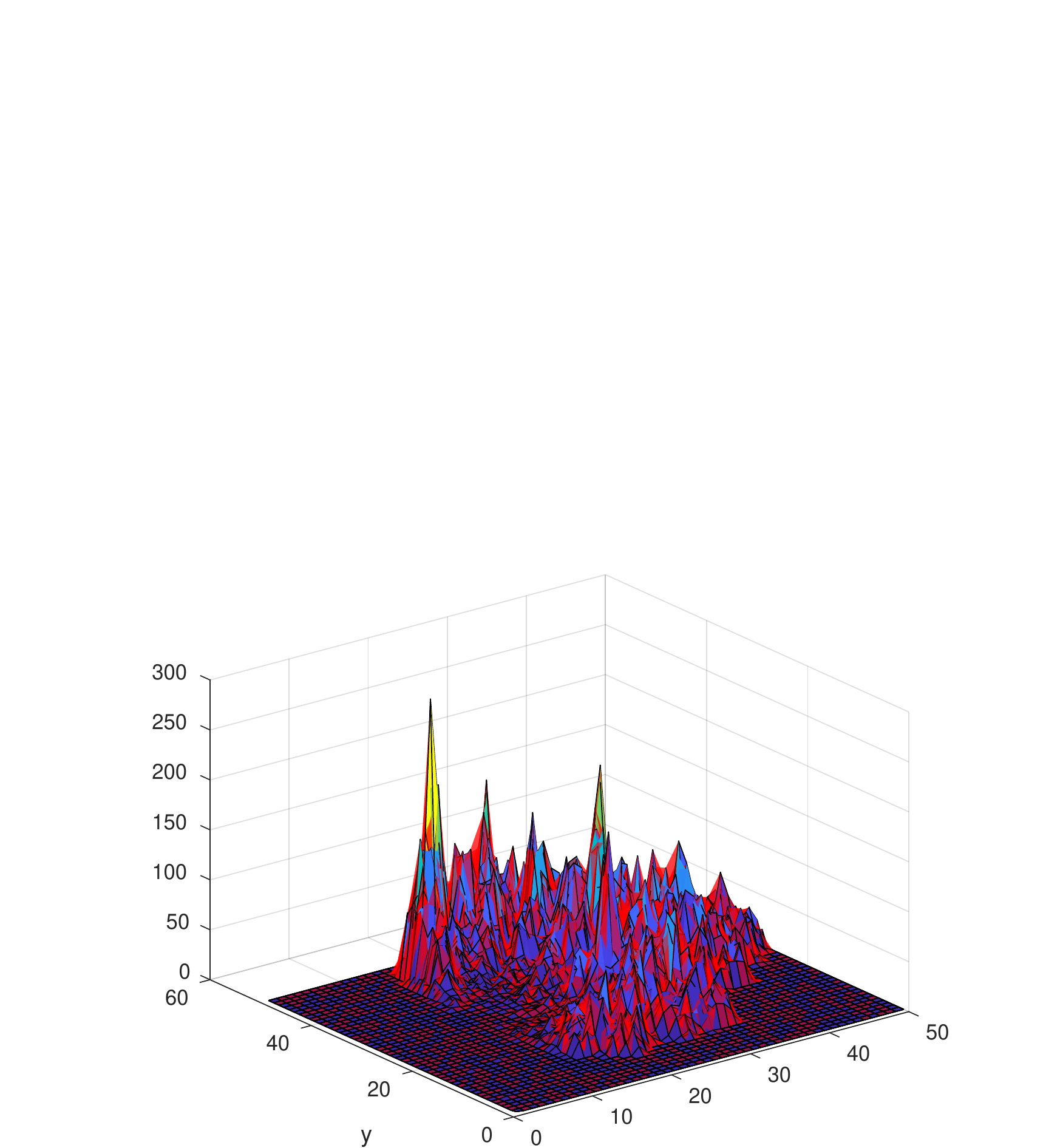}
\caption{An finite-sample histogram approximation 
of the measure over the 2-dimensional state-space 
 generated using $\mu_0$.}
\label{approxNu0}
\end{figure}

\begin{figure}[htp]
\centering 
\includegraphics[trim = 14mm 100mm 18mm 5mm, clip, width=14.35cm]{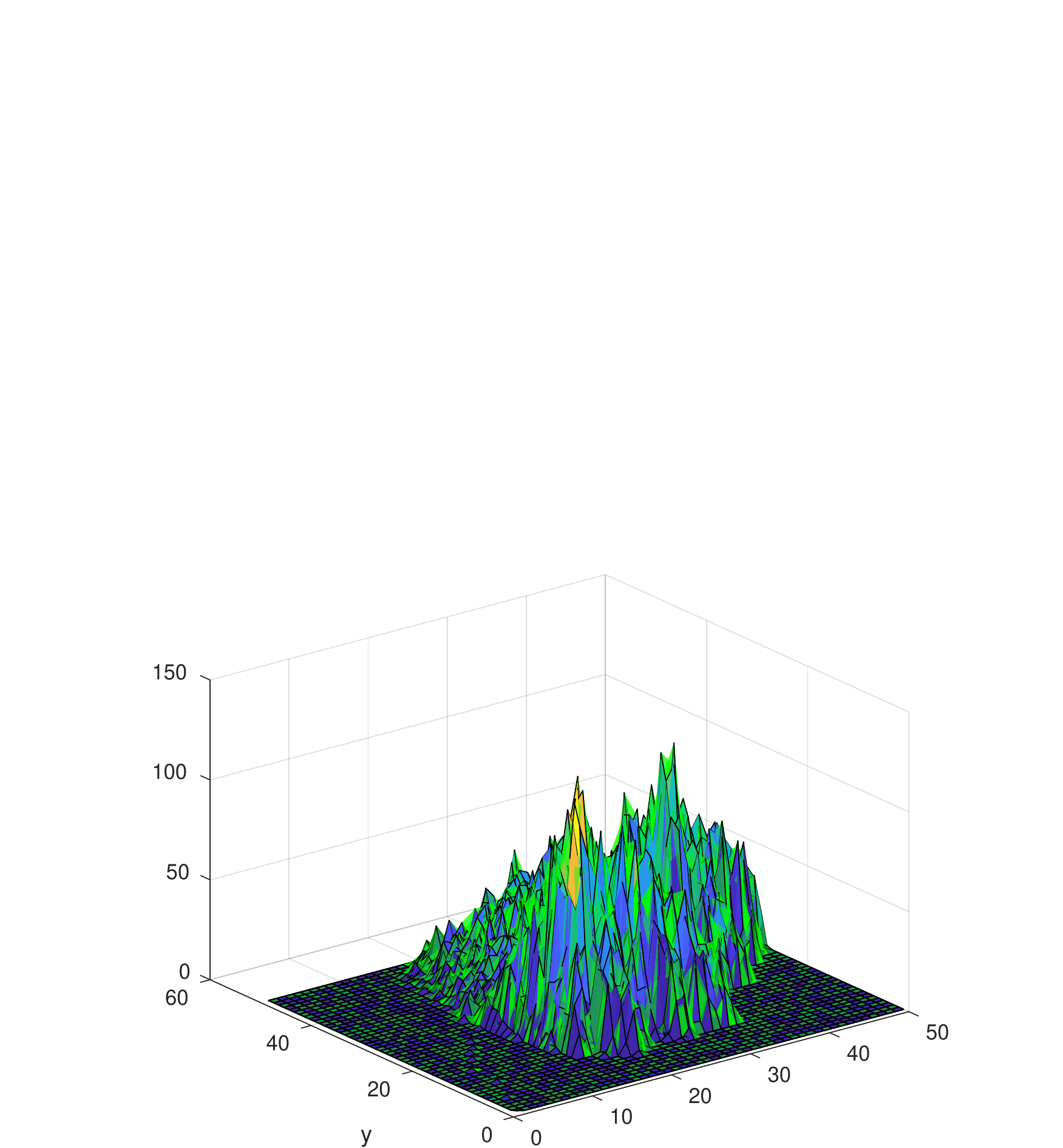}
\caption{A finite-sample histogram approximation 
of the measure over the 2-dimensional state-space 
 generated using $\mu_1$.}
\label{approxNu1}
\end{figure}

\begin{figure}[htp]
\centering 
\includegraphics[trim = 14mm 100mm 18mm 5mm, clip, width=14.35cm]{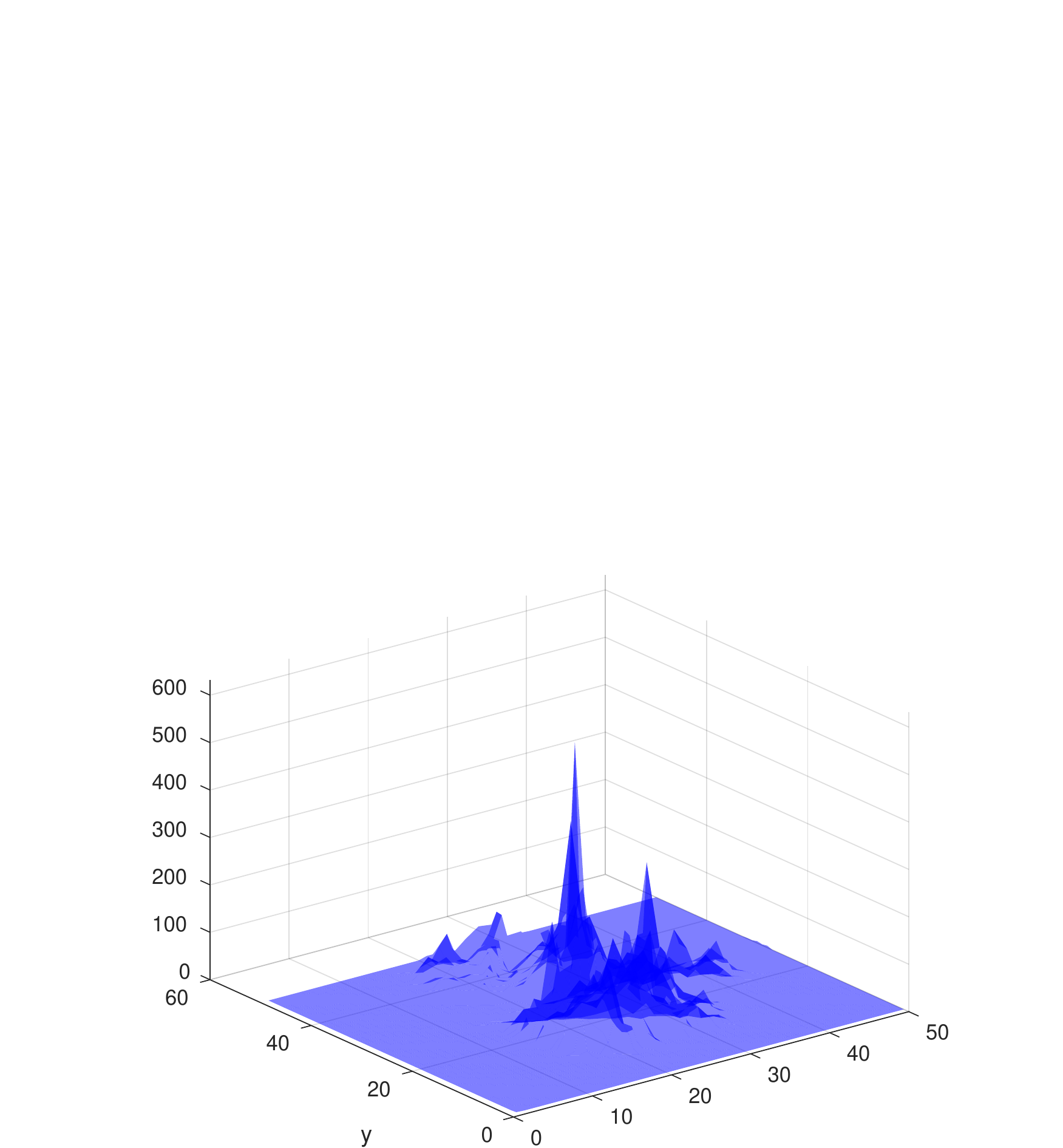}
\caption{A finite-sample histogram approximation 
of the measure over the 2-dimensional state-space 
 generated using $\mu_2$.}
\label{approxNu2}
\end{figure}

\begin{figure}[htp]
\centering 
\includegraphics[trim = 14mm 100mm 18mm 5mm, clip, width=14.35cm]{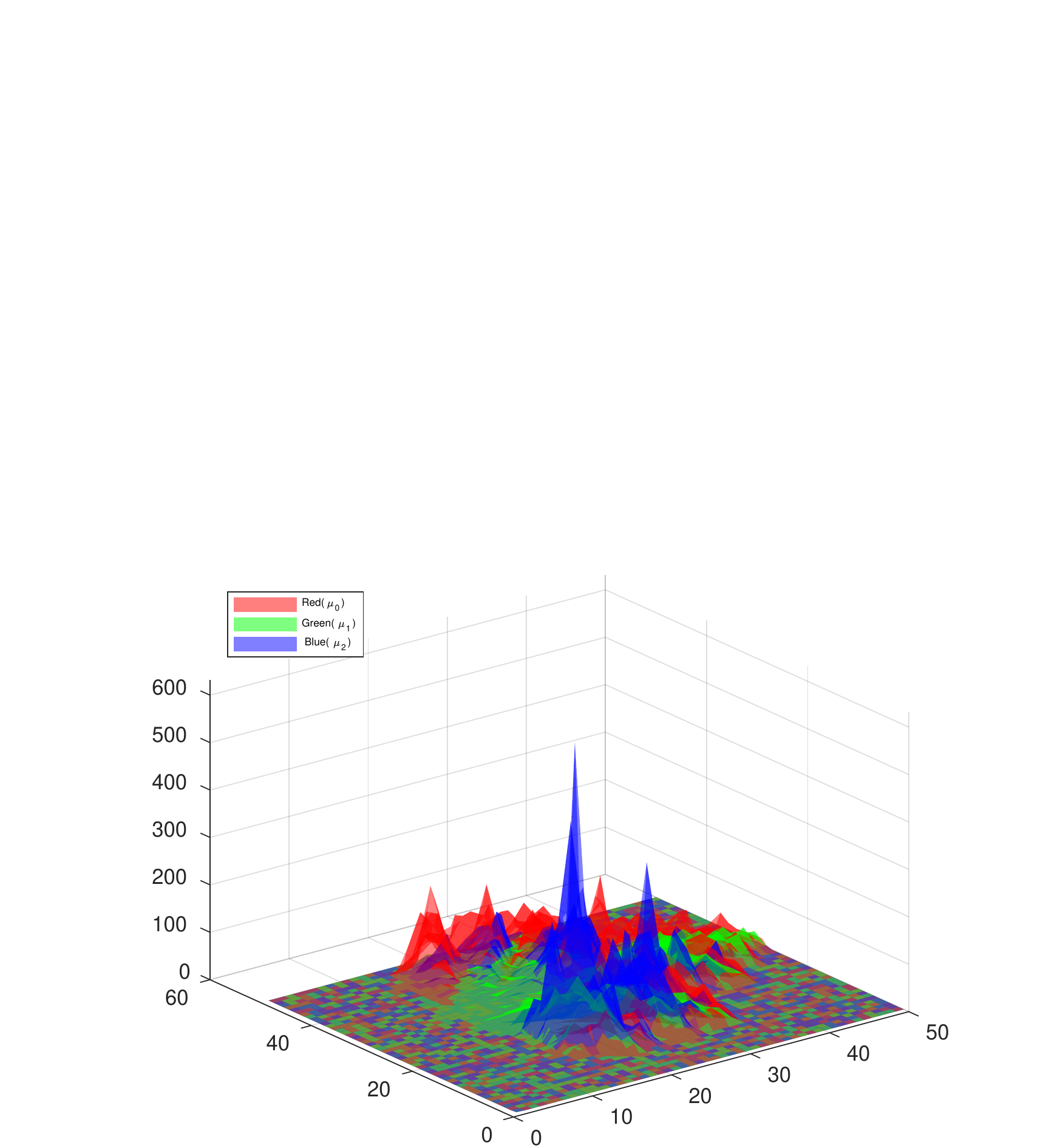}
\caption{A comparison of measures over the 2-dimensional state-space for three measures $\mu_k$ over the family of 4 functions.}
\label{nuCompared}
\end{figure}

\begin{figure}[htp]
\centering 
\includegraphics[trim = 14mm 100mm 18mm 5mm, clip, width=14.35cm]{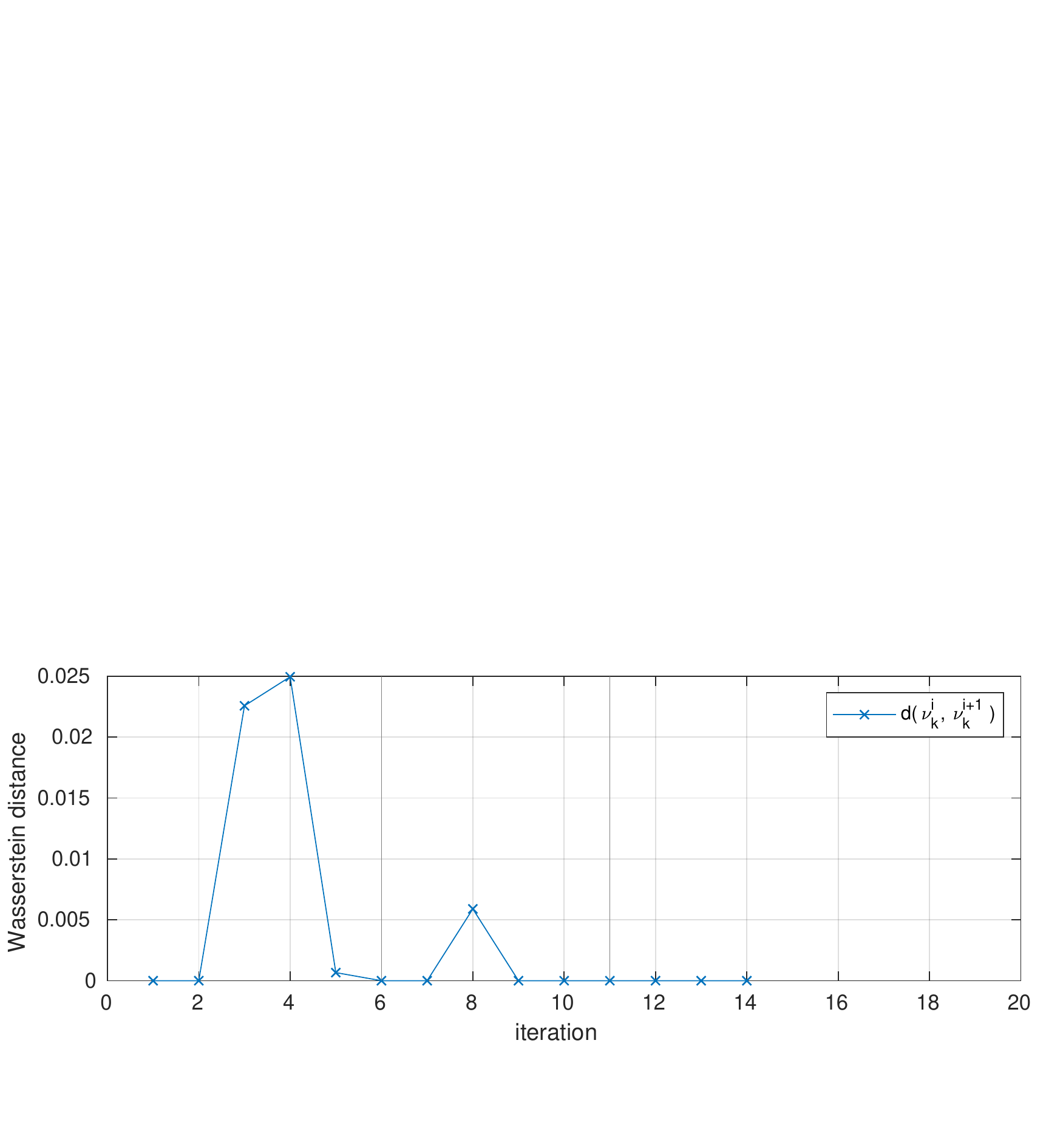}
\caption{Evolution of pairwise distances between subsequent measures over the 2-dimensional state-space.}
\label{nuWasser1}
\end{figure}

\begin{figure}[htp]
\centering 
\includegraphics[trim = 14mm 100mm 18mm 5mm, clip, width=14.35cm]{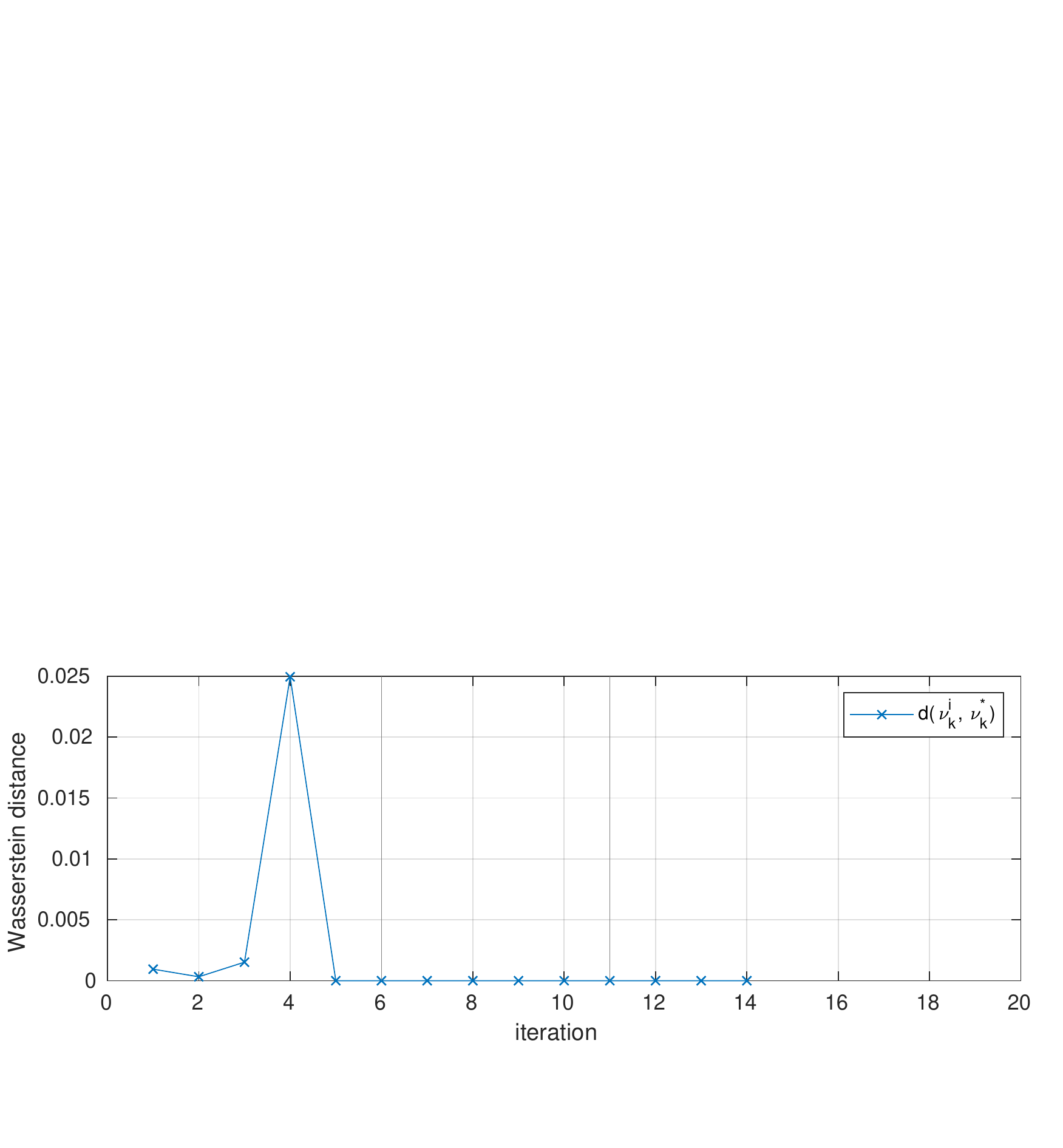}
\caption{Evolution of pairwise distances between a measure over the 2-dimensional state-space and the corresponding invariant measure.}
\label{nuWasser2}
\end{figure}

%\clearpage

\end{document}